\RequirePackage{fix-cm}
\documentclass[smallextended,envcountsame]{svjour3} 
\smartqed  
\usepackage{graphicx, amssymb}
\usepackage{times,a4wide,mathrsfs}
\usepackage{amsmath}
\usepackage{amsfonts}

\newcommand{\C}{\mathbb{C}}

\newcommand{\QQ}{\mathbb{Q}}
\newcommand{\NN}{\mathbb{N}}

\newcommand{\Sy}{\mathfrak S}

\newcommand{\MM}{\mathcal M}

\newcommand{\Bb}{\mathfrak B}
\newcommand{\oo}{\mathfrak o}

\DeclareMathOperator{\gr}{Gr}

\DeclareMathOperator{\ima}{Im}
\newcommand{\rom}{\romannumeral}

\DeclareMathOperator{\aut}{Aut}
\DeclareMathOperator{\ide}{id}

\newtheorem{convention}{Conventions}

\newtheorem{nonumbering}{Theorem}

\newtheorem{nonumberingc}{Corollary}

 \journalname{}

\begin{document}

\title{About Chow groups of certain hyperk\"ahler varieties with non--symplectic automorphisms}

\author{Robert Laterveer}

\institute{CNRS - IRMA, Universit\'e de Strasbourg \at
              7 rue Ren\'e Descartes \\
              67084 Strasbourg cedex\\
              France\\
              \email{{\tt laterv@math.unistra.fr}}   }

\date{Received: date / Accepted: date}

\maketitle

\begin{abstract} Let $X$ be a hyperk\"ahler variety, and let $G$ be a group of finite order non--symplectic automorphisms of $X$. Beauville's conjectural splitting property predicts that each Chow group of $X$ should split in a finite number of pieces.
The Bloch--Beilinson conjectures predict how $G$ should act on these pieces of the Chow groups: certain pieces should be invariant under $G$, while certain other pieces should not contain any non--trivial $G$--invariant cycle. We can prove this for two pieces of the Chow groups when $X$ is the Hilbert scheme of a $K3$ surface and $G$ consists of natural automorphisms. This has consequences for the Chow ring of the quotient $X/G$.
 \end{abstract}



\section{Introduction}

Let $X$ be a hyperk\"ahler variety of dimension $2m$ (i.e., a projective irreducible holomorphic symplectic manifold, cf. \cite{Beau0}, \cite{Beau1}). Let $G\subset\aut(X)$ be a finite cyclic group of order $k$
consisting of non--symplectic automorphisms. We will be interested in the action of $G$ on the Chow groups $A^\ast(X)$.
 (Here, $A^i(X):=CH^i(X)_{\QQ}$ denotes the Chow group of codimension $i$ algebraic cycles modulo rational equivalence with $\QQ$--coefficients.)
Let us suppose $X$ has a multiplicative Chow--K\"unneth decomposition, in the sense of \cite{SV}. This implies the Chow ring of $X$ is a bigraded ring $A^\ast_{(\ast)}(X)$, where each Chow group splits as
  \[ A^i(X) = \bigoplus_j A^i_{(j)}(X)\ ,\]
and the piece $A^i_{(j)}(X)$ is expected to be isomorphic to the graded $\gr^j_F A^i(X)$ for the conjectural Bloch--Beilinson filtration $F^\ast$ on Chow groups. (Conjecturally, all hyperk\"ahler varieties have a multiplicative Chow--K\"unneth decomposition; this is related to Beauville's conjectural splitting property \cite{Beau3}. The existence of a multiplicative Chow--K\"unneth decomposition has been established for Hilbert schemes of $K3$ surfaces \cite{SV}, \cite{V6}, and for generalized Kummer varieties \cite{FTV}.)

Since $H^{2k,0}(X)=H^{2,0}(X)^{\otimes k}$, the group $G$ acts as the identity on $H^{2k,0}(X)$. For $i<2k$, we have that $\sum_{g\in G} g^\ast$ acts as $0$ on $H^{i,0}(X)$. The Bloch--Beilinson conjectures \cite{J2}, combined with the expected isomorphism $A^i_{(j)}(X)\cong\gr^j_F A^i(X)$, thus imply the following conjecture:

\begin{conjecture}\label{theconj} Let $X$ be a hyperk\"ahler variety of dimension $2m$, and let $G\subset\aut(X)$ be a finite cyclic group of order $k$ of non--symplectic automorphisms.
Then
  \[   \begin{split} &A^{2m}_{(j)}(X)\cap A^{2m}(X)^G =\begin{cases}  0\ \ &\hbox{if}\ j< 2k\ ;\\
                                                                           A^{2m}_{(j)}(X)\ \  &\hbox{if}\ j=2k\ ;\\  
                                                                        \end{cases}   \\
                            &A^i_{(i)}(X)\cap A^i(X)^G =\begin{cases}  0\ \ &\hbox{if}\ i< 2k\ ;\\
                                                                           A^i_{(i)}(X)\ \  &\hbox{if}\ i=2k\ ;\\  
                                                                        \end{cases}   \\
                      \end{split}      \]
   \end{conjecture}                        
(Here $A^i(X)^G\subset A^i(X)$ denotes the $G$--invariant part of the Chow group $A^i(X)$.)

The main result in this note is a partial verification of conjecture \ref{theconj} for a certain class of hyperk\"ahler varieties and a certain class of automorphisms.

\begin{nonumbering}[=theorem \ref{main}] Let $S$ be a projective $K3$ surface, and let $X=S^{[m]}$ be the Hilbert scheme of length $m$ subschemes. Let $G\subset\aut(X)$ be a subgroup of order $k$ of natural non--symplectic automorphisms. 
Then
  \[ A^i_{(2)}(X)\cap A^i(X)^G =0\ \ \ \hbox{for}\ i\in\{2,2m\}\ .\]
 \end{nonumbering}

A {\em natural\/} automorphism of $X$ is an automorphism induced by an automorphism of $S$.
Theorem \ref{main} applies to Hilbert schemes of any $K3$ surface $S$ having a finite order non--symplectic automorphism. Such $K3$ surfaces have been intensively studied, and there are lots of 
examples known
\cite{Nik1}, \cite{Nik2}, \cite{Kon}, \cite{Vor}, \cite{LSY}, \cite{Schu}, \cite{Tak}, \cite{AS}, \cite{AST}, \cite{AST2}, \cite{GP}. It would be interesting to prove theorem \ref{main} also for non--symplectic automorphisms that are {\em non--natural\/}; this seems considerably more difficult (cf. \cite{BlochHK4} for one special case where theorem \ref{main} is proven for a non--natural involution).

Theorem \ref{main} has interesting consequences for the Chow ring of the quotient:

\begin{nonumberingc}[=corollaries \ref{cor1} and \ref{cor2}] Let $X$ and $G$ be as in theorem \ref{main}, and let $Y:=X/G$. 

\noindent
(\rom1) Let $a\in A^{2m-1}(Y)$ be a $1$--cycle which is in the image of the intersection product map
  \[ A^{i_1}(Y)\otimes A^{i_2}(Y)\otimes\cdots\otimes A^{i_r}(Y)\ \to\ A^{2m-1}(Y)\ ,\]
  where all $i_j$ are $\le 2$. Then $a$ is rationally trivial if and only if $a$ is homologically trivial.

\noindent
(\rom2) Let $a\in A^{2m}(Y)$ be a $0$--cycle which is in the image of the intersection product map
  \[ A^3(Y)\otimes A^{i_1}(Y)\otimes\cdots\otimes A^{i_r}(Y)\ \to\ A^{2m}(Y)\ ,\]
   where all $i_j$ are $\le 2$. Then $a$ is rationally trivial if and only if $a$ is homologically trivial.
  \end{nonumberingc}

These corollaries illustrate the following expectation: for certain special varieties with a multiplicative Chow--K\"unneth decomposition, the subring $A^\ast_{(0)}\subset A^\ast$ on which the cycle class map is injective should be larger than for hyperk\"ahler varieties. Indeed, for a quotient $Y=X/G$ where $X$ is hyperk\"ahler and $G\subset\aut(X)$ is a finite order group of non--symplectic automorphisms, one expects that codimension $2$ cycles lie in $A^\ast_{(0)}(Y)$.

Results similar in spirit have been obtained for certain other hyperk\"ahler varieties and their Calabi--Yau quotients in \cite{EPW}, \cite{LSYmoi}, \cite{BlochHK4}.

\begin{convention} In this article, the word {\sl variety\/} will refer to a reduced irreducible scheme of finite type over $\C$. A {\sl subvariety\/} is a (possibly reducible) reduced subscheme which is equidimensional. 

{\bf All Chow groups will be with rational coefficients}: we will denote by $A_j(X)$ the Chow group of $j$--dimensional cycles on $X$ with $\QQ$--coefficients; for $X$ smooth of dimension $n$ the notations $A_j(X)$ and $A^{n-j}(X)$ are used interchangeably. 


The notations $A^j_{hom}(X)$, $A^j_{AJ}(X)$ will be used to indicate the subgroups of homologically trivial, resp. Abel--Jacobi trivial cycles.
For a morphism $f\colon X\to Y$, we will write $\Gamma_f\in A_\ast(X\times Y)$ for the graph of $f$.
The contravariant category of Chow motives (i.e., pure motives with respect to rational equivalence as in \cite{Sc}, \cite{MNP}) will be denoted $\MM_{\rm rat}$.



We will write $H^j(X)$ 
to indicate singular cohomology $H^j(X,\QQ)$.

Given a group $G\subset\aut(X)$ of automorphisms of $X$, we will write $A^j(X)^G$ (and $H^j(X)^G$) for the subgroup of $A^j(X)$ (resp. $H^j(X)$) invariant under $G$.
\end{convention}

\section{Preliminary}

\subsection{Quotient varieties}

\begin{definition} A {\em projective quotient variety\/} is a variety
  \[ Y=X/G\ ,\]
  where $X$ is a smooth projective variety and $G\subset\aut(X)$ is a finite group.
  \end{definition}
  
 \begin{proposition}[Fulton \cite{F}]\label{quot} Let $Y$ be a projective quotient variety of dimension $n$. Let $A^\ast(Y)$ denote the operational Chow cohomology ring. The natural map
   \[ A^i(Y)\ \to\ A_{n-i}(Y) \]
   is an isomorphism for all $i$.
   \end{proposition}
   
   \begin{proof} This is \cite[Example 17.4.10]{F}.
      \end{proof}

\begin{remark} It follows from proposition \ref{quot} that the formalism of correspondences goes through unchanged for projective quotient varieties (this is also noted in \cite[Example 16.1.13]{F}). We can thus consider motives $(Y,p,0)\in\MM_{\rm rat}$, where $Y$ is a projective quotient variety and $p\in A^n(Y\times Y)$ is a projector. For a projective quotient variety $Y=X/G$, one readily proves (using Manin's identity principle) that there is an isomorphism
  \[  h(Y)\cong h(X)^G:=(X,\Delta^G_X,0)\ \ \ \hbox{in}\ \MM_{\rm rat}\ ,\]
  where $\Delta^G_X$ denotes the idempotent 
 \[ \Delta^G_X:=  {1\over \vert G\vert}{\sum_{g\in G}}\ \Gamma_g\ \ \ \in A^n(X\times X).\] 
 (NB: $\Delta^G_X$ is a projector on the $G$--invariant part of the Chow groups $A^\ast(X)^G$.) 
  \end{remark}

\subsection{MCK decomposition}

\begin{definition}[Murre \cite{Mur}] Let $X$ be a projective quotient variety of dimension $n$. We say that $X$ has a {\em CK decomposition\/} if there exists a decomposition of the diagonal
   \[ \Delta_X= \pi_0+ \pi_1+\cdots +\pi_{2n}\ \ \ \hbox{in}\ A^n(X\times X)\ ,\]
  such that the $\pi_i$ are mutually orthogonal idempotents and $(\pi_i)_\ast H^\ast(X)= H^i(X)$. A CK decomposition is {\em self--dual\/} if $\pi_i={}^t \pi_{2n-i}$ in $A^n(X\times X)$
  for all $i$ (here ${}^t \pi$ denotes the transpose of $\pi$).
  
  (NB: ``CK decomposition'' is shorthand for ``Chow--K\"unneth decomposition''.)
\end{definition}

\begin{remark} The existence of a CK decomposition for any smooth projective variety is part of Murre's conjectures \cite{Mur}, \cite{J2}. 
\end{remark}

\begin{definition}[Shen--Vial \cite{SV}] Let $X$ be a projective quotient variety of dimension $n$. Let $\Delta_X^{sm}\in A^{2n}(X\times X\times X)$ be the class of the small diagonal
  \[ \Delta_X^{sm}:=\bigl\{ (x,x,x)\ \vert\ x\in X\bigr\}\ \subset\ X\times X\times X\ .\]
  An MCK decomposition is a CK decomposition $\{\pi_i\}$ of $X$ that is {\em multiplicative\/}, i.e. it satisfies
  \[ \pi_k\circ \Delta_X^{sm}\circ (\pi_i\times \pi_j)=0\ \ \ \hbox{in}\ A^{2n}(X\times X\times X)\ \ \ \hbox{for\ all\ }i+j\not=k\ .\]
  
 (NB: ``MCK decomposition'' is shorthand for ``multiplicative Chow--K\"unneth decomposition''.) 
  \end{definition}
  
  \begin{remark} The small diagonal (seen as a correspondence from $X\times X$ to $X$) induces the {\em multiplication morphism\/}
    \[ \Delta_X^{sm}\colon\ \  h(X)\otimes h(X)\ \to\ h(X)\ \ \ \hbox{in}\ \MM_{\rm rat}\ .\]
 Suppose $X$ has a CK decomposition
  \[ h(X)=\bigoplus_{i=0}^{2n} h^i(X)\ \ \ \hbox{in}\ \MM_{\rm rat}\ .\]
  By definition, this decomposition is multiplicative if for any $i,j$ the composition
  \[ h^i(X)\otimes h^j(X)\ \to\ h(X)\otimes h(X)\ \xrightarrow{\Delta_X^{sm}}\ h(X)\ \ \ \hbox{in}\ \MM_{\rm rat}\]
  factors through $h^{i+j}(X)$.
  It follows that if $X$ has an MCK decomposition, then setting
    \[ A^i_{(j)}(X):= (\pi^X_{2i-j})_\ast A^i(X) \ ,\]
    one obtains a bigraded ring structure on the Chow ring: that is, the intersection product sends $A^i_{(j)}(X)\otimes A^{i^\prime}_{(j^\prime)}(X) $ to  $A^{i+i^\prime}_{(j+j^\prime)}(X)$.
    
      It is expected (but not proven !) that for any $X$ with an MCK decomposition, one has
    \[ A^i_{(j)}(X)\stackrel{??}{=}0\ \ \ \hbox{for}\ j<0\ ,\ \ \ A^i_{(0)}(X)\cap A^i_{hom}(X)\stackrel{??}{=}0\ ;\]
    this is related to Murre's conjectures B and D \cite{Mur}.

  The property of having an MCK decomposition is severely restrictive, and is closely related to Beauville's ``(weak) splitting property'' \cite{Beau3}. For more ample discussion, and examples of varieties with an MCK decomposition, we refer to \cite[Section 8]{SV}, as well as \cite{V6}, \cite{SV2}, \cite{FTV}, \cite{EPW}.
    \end{remark}
    
\begin{theorem}[Vial \cite{V6}]\label{hilbk} Let $S$ be an algebraic $K3$ surface, and let $X=S^{[m]}$ be the Hilbert scheme of length $m$ subschemes of $S$. Then $X$ has a self--dual MCK decomposition. One has
  \[ A^i_{(j)}(X)=0\ \ \ \hbox{for\ all\ $j$\ odd\ and\ for\ all\ $j>i$}\ .\]
\end{theorem}

\begin{proof} This is \cite[Theorem 1]{V6}. For later use, we briefly review the construction. First, one takes an MCK decomposition $\{\Pi^S_i\}$ for $S$ (this exists, thanks to \cite{SV}). Taking products, this induces an MCK decomposition $\{ \Pi^{S^r}\}$ for $S^r$, $r\in\NN$. This product MCK decomposition is invariant under the action of the symmetric group $\Sy_r$, and hence it induces an MCK decomposition
$\{ \Pi^{S^{(r)}}\}$ for the symmetric products $S^{(r)}$, $r\in\NN$.
There is the isomorphism of de Cataldo--Migliorini \cite{CM}
  \[  \bigoplus_{\mu\in \Bb(m)} ({}^t \hat{\Gamma}_\mu)_\ast \colon\ \ \ A^i(X)\ \xrightarrow{\cong}\ \bigoplus_{\mu\in \Bb(m)}\ A^{i+l(\mu)-m}(S^{(\mu)})\ ,\]
  where $\Bb(m)$ is the set of partitions of $m$, $l(\mu)$ is the length of the partition $\mu$, and $S^{(\mu)}=S^{l(\mu)}/ {\Sy_{l(\mu)}}$, and ${}^t \hat{\Gamma}_\mu$ is a correspondence in $A^{m+l(\mu)}(S^{[m]}\times S^{(\mu)})$. Using this isomorphism, Vial defines \cite[Equation (4)]{V6} a natural CK decomposition for $X$, by setting
  \begin{equation}\label{defv} \Pi_i^X:= {\displaystyle\sum_{\mu\in \Bb(m)}} \ {1\over m_\mu} \hat{\Gamma}_\mu \circ \Pi^{S^{(\mu)}}_{i-2m+2l(\mu)}\circ {}^t \hat{\Gamma}_\mu\  , \end{equation}
  where the $m_\mu$ are rational numbers coming from the de Cataldo--Migliorini isomorphism.
  The $\{\Pi^X_i\}$ of definition (\ref{defv}) are proven to be an MCK decomposition.
  
The self--duality of the $\{ \Pi^X_i\}$ is apparent from definition (\ref{defv}). The fact that $A^i_{(j)}(X)$ vanishes for $j$ odd is because $\Pi^X_j=0$ for $j$ odd. The vanishing for $j>i$ 
follows from the fact that by construction, the projector $\Pi^X_\ell$ is supported on $V\times X$ with $\dim V=\ell$; this implies (for reasons of dimension) that
  \[ (\Pi^X_\ell)_\ast A^i(X)=0\ \ \ \hbox{for\ all\ }\ell<i\ .\] 
 \end{proof}

\begin{remark}\label{compat} It follows from definition (\ref{defv}) that the de Cataldo--Migliorini isomorphism is compatible with the bigrading of the Chow ring, in the sense that there are induced isomorphisms
  \[   \bigoplus_{\mu\in \Bb(m)} ({}^t \hat{\Gamma}_\mu)_\ast \colon\ \ \ A^i_{(j)}(X)\ \xrightarrow{\cong}\ \bigoplus_{\mu\in \Bb(m)}\ A^{i+l(\mu)-m}_{(j)}(S^{(\mu)})\ .\]
  In particular, there are split injections
   \[   \bigoplus_{\mu\in \Bb(m)} ({}^t {\Gamma}_\mu)_\ast \colon\ \ \ A^i_{(j)}(X)\ \hookrightarrow\ \bigoplus_{\mu\in \Bb(m)}\ A^{i+l(\mu)-m}_{(j)}(S^{\mu})\ .\] 
   (Here, the right--hand side refers to the product MCK decomposition of $S^\mu$.) 
  \end{remark}

\begin{lemma}[Shen--Vial]\label{diag} Let $X$ be a projective quotient variety of dimension $n$, and suppose $X$ has a self--dual MCK decomposition. Then
  \[ \begin{split} &\Delta_X\ \ \in A^n_{(0)}(X\times X)\ ,\\
                         &\Delta_X^{sm}\ \ \in A^{2n}_{(0)}(X\times X\times X)\ .\\
                \end{split} \]        
\end{lemma}

\begin{proof} The first statement follows from \cite[Lemma 1.4]{SV2} when $X$ is smooth. The same argument works for projective quotient varieties; the point is just that
  \[ \begin{split} \Delta_X &= {\displaystyle\sum_{i=0}^{2n}}\Pi_{i}^X =   {\displaystyle\sum_{i=0}^{2n}}\Pi_{i}^X\circ \Pi_i^X\\
                                        &=   {\displaystyle\sum_{i=0}^{2n}}({}^t\Pi_{i}^X \times  \Pi_i^X)_\ast \Delta_X\\
                                        &=  {\displaystyle\sum_{i=0}^{2n}}(\Pi_{2n-i}^X \times  \Pi_i^X)_\ast \Delta_X\\
                                        &=  (\Pi_{2n}^{X\times X})_\ast \Delta_X\ \ \ \in A^n_{(0)}(X\times X)\ .\\
                                    \end{split}\]
                    (Here, the second line follows from Lieberman's lemma \cite[Lemma 3.3]{V3}, and the last line is the fact that the product of $2$ MCK decompositions is MCK
                    \cite[Theorem 8.6]{SV}.)
                    
   The second statement is proven for smooth $X$ in \cite[Proposition 8.4]{SV}; the same argument works for projective quotient varieties.                             
       \end{proof}

\subsection{MCK for products}

  \begin{proposition}\label{prod} Let $S$ be a $K3$ surface. There exist correspondences
    \[  \Theta_1\ ,\ldots,\ \Theta_m\in A^{2m}(S^{m}\times S)\ ,\ \ \ \Xi_1\ ,\ldots, \ \Xi_m\in  A^{2}(S\times S^{m})  \]
    such that the composition
    \[  \begin{split}   A^{2m}_{(2)}(S^m)\ \xrightarrow{((\Theta_1)_\ast,\ldots, (\Theta_m)_\ast)}\
             A^2(S)\oplus \cdots \oplus A^2(S)&\\
             \ \ \ \ \ \ \xrightarrow{(\Xi_1)_\ast+\ldots+(\Xi_m)_\ast}\ A^{2m}(S^m)&\\
             \end{split} \]
      is the identity.
    \end{proposition}
    
    \begin{proof} 
    By construction \cite{SV}, the MCK decomposition for $S$ is given by
      \[  \Pi_0^S=\oo_S\times S\ ,\ \ \ \Pi_4^S=S\times\oo_S\ ,\ \ \ \Pi_2^S=\Delta_S-\pi_0^S-\pi_4^S\ .\]
Here $\oo_S\in A^2(S)$ denotes the ``distinguished point'' of \cite{BV} (any point lying on a rational curve in $S$ equals $\oo_S$ in $A^2(S)$).
   Let 
    \[ p_{i,j}\colon\ \ \  S^{2m}\ \to\ S^{2} \ \ \ (1\le i<j\le 2m)\]
    denote projection to the $i$-th and $j$-th factor, and let 
    \[ p_i\colon \ \ \  S^{m}\ \to\ S \ \ \  (1\le i\le m) \]
    denote projection to the $i$--th factor.
            
        We now claim that there is equality
       \begin{equation}\label{both} \begin{split}  \Pi_{4m-2}^{S^{m}} =  \Bigl(  {}^t \Gamma_{p_{1}}\circ \Pi_2^S\circ \Gamma_{p_{1}}\circ 
               \bigl(   (p_{1,m+1})^\ast (\Delta_S )\cdot \prod_{\stackrel{2\le j\le 2m}{j\not=m+1}} (p_{j})^\ast (\oo_S) &\bigr)\\  
               + \ldots +
                           {}^t \Gamma_{p_{m}}\circ \Pi_2^S\circ \Gamma_{p_{m}}\circ 
               \bigl(   (p_{m,2m})^\ast (\Delta_S )\cdot \prod_{\stackrel{1\le j\le 2m-1}{j\not=m}}(p_{j})^\ast (\oo_{S})    \bigr)   \Bigr)\\
               \ \ \ \hbox{in}\ A^{2m}(&S^{m}\times S^{m})\ .\\
               \end{split}\end{equation}
         Indeed, using Lieberman's lemma \cite[Lemma 3.3]{V3}, we find that
         \[ \begin{split}  {}^t \Gamma_{p_{1}}\circ &\Pi_2^S\circ \Gamma_{p_{1}} = ({}^t \Gamma_{p_{1,m+1}})_\ast 
         (\Pi_2^{S})=
           (p_{1,m+1})^\ast (\Pi_2^{S})\ ,\\
                           &\vdots\\
            {}^t \Gamma_{p_{m}}\circ &\Pi_2^S\circ \Gamma_{p_{m}} = ({}^t \Gamma_{p_{m,2m}})_\ast 
        ( \Pi_2^{S}) =
           (p_{m,2m})^\ast (\Pi_2^{S})\ .\\
           \end{split}           \]
           
       Let us now (by way of example) consider the first summand of the right--hand--side of (\ref{both}). For brevity, let
        \[ P\colon\ \ \  S^{3m}\ \to\ S^{2m} \]
        denote the projection on the first $m$ and last $m$ factors. Writing out the definition of composition of correspondences,
        we find that
       \[ \begin{split}     &  {}^t \Gamma_{p_{1}}\circ \Pi_2^S\circ \Gamma_{p_{1}}\circ 
               \bigl(   (p_{1,m+1})^\ast (\Delta_S )\cdot \prod_{\stackrel{2\le j\le 2m}{j\not=m+1}} (p_{j})^\ast (\oo_{S}) \bigr) =\\
                &    \bigl((p_{1,m+1})^\ast (\Pi_2^{S})\bigr)   \circ 
               \bigl(   (p_{1,m+1})^\ast (\Delta_{S} )\cdot \prod_{\stackrel{2\le j\le 2m}{j\not=m+1}} (p_{j})^\ast (\oo_{S}) \bigr) =\\ 
               & P_\ast    \Bigl( \bigl( (\Delta_{S})_{(1,m+1)} \times {\mathfrak o}_{S} \times\cdots\times{\mathfrak o}_{S} \times S\times\cdots\times S\bigr)\cdot \\
               &\ \ \ \ \ \ \ \ \bigl( S\times\cdots\times S\times (\Pi_2^{S})_{(m+1,2m+1)}\times S\times\cdots\times S     \bigr)  \Bigr)= \\
               & P_\ast \Bigl(  \bigl((\Delta_{S}\times S)\cdot (S\times\Pi_2^{S})\bigr)_{(1,m+1,2m+1)}\times  {\mathfrak o}_{S}\times\cdots\times {\mathfrak o}_{S}\times S\times\cdots\times S\Bigr)=\\
               & \Pi_2^{S}\times \Pi_4^{S}\times\cdots\times \Pi_4^{S}\ \ \ \ \ \ \hbox{in}\ A^{2m} (S^m\times S^m)\ .\\
               \end{split}\]  
               (Here, we use the notation $(C)_{(i, j)}$ to indicate that the cycle $C$ lies in the $i$th and $j$th factor, and likewise for $(D)_{(i,j,k)}$.)           
                          
      Doing the same for the other summands in (\ref{both}), one convinces oneself that both sides of (\ref{both}) are equal to the product Chow--K\"unneth component
         \[  \Pi_{4m-2}^{S^m}=\Pi_2^{S}\times \Pi_4^{S}\times  \cdots\times\Pi_4^{S}  +\cdots  +      \Pi_4^{S}\times\cdots\times\Pi_4^{S}\times \Pi_2^{S}    \ \ \ \in A^{2m}(S^m\times S^m)\ ,\]
         thus proving the claim.

  Let us now define
      \[ \begin{split}
           \Theta_i&:=  \Gamma_{p_{i}}\circ 
               \bigl(   (p_{i,m+i})^\ast (\Delta_S )\cdot \prod_{\stackrel{j\in [1,2m]}{ j\not\in\{i,m+i\}}} (p_{j})^\ast (\oo_{S})    \bigr)\ \ \ \in A^{2m}(S^{m}\times S)\ ,\\  
                 \Xi_i&:= {}^t \Gamma_{p_{i}}\circ \Pi_2^S\ \ \ \ \ \ \in A^2(S\times S^{m}) \ ,\\
                       \end{split}\]
                       where $1\le i\le m$.
    It follows from equation (\ref{both}) that there is equality 
      \begin{equation}\label{transp}  \bigl( \Xi_1\circ \Theta_1 + \cdots +\Xi_m\circ \Theta_m   \bigr){}_\ast =
      \bigl(\Pi_{4m-2}^{S^m}\bigr){}_\ast\colon
           \  A^{i}_{(j)}(S^m\bigr)\ \to\ A^{i}_{(j)}(S^m\bigr)\  \ \ \forall (i,j)\ .
      \end{equation}      
      Taking $(i,j)=(2m,2)$, this proves the proposition.      
         \end{proof}

 The following is a version of proposition \ref{prod} for the group $A^2_{(2)}(S^m)$:     
  
 \begin{proposition}\label{prod3} Let $S$ be a $K3$ surface. There exist correspondences
    \[  {}^t\Theta_1\ ,\ldots,\ {}^t\Theta_m\in A^{2m}(S\times S^{m}))\ ,\ \ \ {}^t \Xi_1\ ,\ldots, \ {}^t\Xi_m\in  A^{2}( S^{m}\times S)  \]
    such that the composition
    \[  \begin{split}   A^{2}_{(2)}(S^m)\ \xrightarrow{(({}^t\Xi_1\vert_{S^{m+1}})_\ast,\ldots, ({}^t\Xi_m\vert_{S^{m+1}})_\ast)}\
             A^2(S)\oplus \cdots \oplus A^2(S)&\\
             \ \ \ \ \ \ \xrightarrow{(({}^t\Theta_1+\ldots+{}^t\Theta_m)\vert_{S^{m+1}})_\ast}\ A^{2}(S^m\bigr)&\\
             \end{split} \]
      is the identity.
  \end{proposition} 
  
  \begin{proof} 
       By construction, the product MCK decomposition $\{ \Pi_i^{S^m}\}$ satisfies
          \[ \Pi_2^{S^m} = {}^t \bigl(\Pi_{4m-2}^{S^m}\bigr)\ \ \ \hbox{in}\ A^{2m}( S^m\times S^m)\ .\]   
     Hence, the transpose of equation (\ref{transp}) gives the equality
        \[       \bigl( \Pi_2^{S^m} \bigr){}_\ast = \bigl( {}^t (\Pi_{4m-2}^{S^m})\bigr){}_\ast = \bigl( {}^t \Theta_1\circ {}^t \Xi_1+\ldots+{}^t \Theta_m\circ {}^t \Xi_m\bigr){}_\ast\colon\  \  \ A^{i}_{(j)}(S^m\bigr)\ \to\ A^{i}_{(j)}(S^m\bigr)\ \ \ \ \forall (i,j)\ .
        \]
        Taking $(i,j)=(2,2)$, this proves the proposition.
         \end{proof}

\subsection{Birational invariance}

\begin{proposition}[Rie\ss \cite{Rie}, Vial \cite{V6}]\label{birat} Let $X$ and $X^\prime$ be birational hyperk\"ahler varieties. Assume $X$ has an MCK decomposition. Then also $X^\prime$ has an MCK decomposition, and there are natural isomorphisms
  \[ A^i_{(j)}(X)\cong A^i_{(j)}(X^\prime)\ \ \ \hbox{for\ all\ }i,j\ .\]
\end{proposition}

\begin{proof} As noted by Vial \cite[Introduction]{V6}, this is a consequence of Rie\ss's result that $X$ and $X^\prime$ have isomorphic Chow motive (as algebras in the category of Chow motives). For more details, cf. \cite[Section 6]{SV} or \cite[Lemma 2.8]{BlochHK4}.
\end{proof}

\subsection{A commutativity lemma}

\begin{lemma}\label{comm} Let $S$ be an algebraic $K3$ surface, and let $\{\Pi_i^S\}$ be the MCK decomposition as above. Let $h\in\aut(S)$.
Then
  \[ \Gamma_h\circ \Pi_i^S=\Pi_i^S\circ \Gamma_h\ \ \ \hbox{in}\ A^2(S\times S)\ \ \ \forall i\ .\]
\end{lemma}

\begin{proof} It suffices to prove this for $i=0$. Indeed, by definition of $\{\Pi_i^S\}$ we have
  \[ \begin{split} \Pi_4^S&:={}^t \Pi_0^S\ \ \ \hbox{in}\ A^2(S\times S)\ ,\\
                      \Pi_2^S&:=\Delta_S-\Pi_0^S-\Pi_4^S\ .\\
                      \end{split}\]
        Supposing the lemma holds for $i=0$, by taking transpose correspondences we get an equality
        \[ \Gamma_{h^{-1}}\circ \Pi_4^S= \Pi_4^S\circ \Gamma_{h^{-1}}\ \ \ \hbox{in}\ A^2(S\times S)\ .\]      
        Composing on both sides with $\Gamma_h$, we get
        \[ \Pi_4^S\circ \Gamma_h=\Gamma_h\circ \Pi_4^S\ \ \ \hbox{in}\ A^2(S\times S)\ .\]      
        Next, since obviously the diagonal $\Delta_S$ commutes with $\Gamma_h$, we also get
        \[  \Gamma_h\circ \Pi_2^S=   \Gamma_h\circ (\Delta_S-\Pi_0^S-\Pi_4^S)=  (\Delta_S-\Pi_0^S-\Pi_4^S)\circ \Gamma_h=   \Pi_2^S\circ \Gamma_h\ \ \ \hbox{in}\ A^2(S\times S)      \ .\]
        
     It remains to prove the lemma for $i=0$. The projector $\Pi_0^S$ is defined as
     \[ \Pi_0^S=\oo_S\times S\ \ \ \in A^2(S\times S)\ ,\]
     where $\oo_S\in A^2(S)$ is the ``distinguished point'' of \cite{BV}. 
     Let $x\in S$ be a point lying on a rational curve. Then $h^\ast(\oo_S)=h^{-1}(x)$ is again a point lying on a rational curve, and so
          \[ h^\ast(\oo_S)= \oo_S\ \ \ \hbox{in}\ A^2(S)\ .\]
     
     Using Lieberman's lemma \cite[Lemma 3.3]{V3}, we find that
     \[  \begin{split} \Pi_0^S\circ \Gamma_h &=  ({}^t \Gamma_h\times\Delta_S)_\ast (\Pi_0^S)\\ &= ({}^t \Gamma_h\times\Delta_S)_\ast(\oo_S\times S) \\ &= h^\ast(\oo_S)\times S \\ &=\oo_S\times S=\Pi_0^S\ \ \ \hbox{in}\ A^2(S\times S)\ ,\\
     \end{split}\]
     whereas obviously
     \[ \Gamma_h\circ \Pi_0^S = (\Delta_S\times \Gamma_h)_\ast (\oo_S\times S)=\oo_S\times S=\Pi_0^S\ \ \ \hbox{in}\ A^2(S\times S)\ .\]
     This proves the $i=0$ case of the lemma.
      \end{proof}
     
  The following lemmas establish some corollaries of lemma \ref{comm}:  
  
      \begin{lemma}\label{idemp} Let $S$ be an algebraic $K3$ surface, and $G_S\subset\aut(S)$ a group of finite order $k$. For any $r\in\NN$, let $\{\Pi_i^{S^r}\}$ denote the product MCK decomposition of $S^r$ induced by the MCK decomposition of $S$ as above. Let
  \[ \Delta^G_{S^r}:={1\over k}{\displaystyle\sum_{g\in G_S}}\ \Gamma_g\times\cdots\times\Gamma_g\ \ \ \in A^{2r}(S^r\times S^r)\ .\]
  Then
  \[ \Delta^G_{S^r}\circ \Pi_i^{S^r}= \Pi_i^{S^r}\circ \Delta^G_{S^r}\ \ \ \in A^{2r}(S^r\times S^r)\ \]
  is an idempotent, for any $i$.
  \end{lemma}
  
 \begin{proof} It suffices to prove the commutativity statement. (Indeed, since both $\Delta^G_{S^r}$ and $\Pi_i^{S^r}$ are idempotent, the idempotence of their composition follows immediately from the stated commutativity relation.) 
  To prove the commutativity statement, we will prove more precisely that for any $h\in\aut(S)$ we have equality
    \begin{equation}\label{eachg}  \Gamma_{h^{\times r}}\circ \Pi_i^{S^r}= \Pi_i^{S^r} \circ \Gamma_{h^{\times r}}  \ \ \ \in A^{2r}(S^r\times S^r)\ .  \end{equation}
  This can be seen as follows:
  we have
    \[ \begin{split}  \Gamma_{h^{\times r}}\circ \Pi_i^{S^r} &= (\Gamma_h\times \cdots \times\Gamma_h)\circ ({\displaystyle\sum_{i_1+\cdots+i_r=i}} \pi_{i_1}^S\times \cdots\times \pi_{i_r}^S)\\
                                                           &=   {\displaystyle\sum_{i_1+\cdots+i_r=i}} (\Gamma_h\circ \Pi^S_{i_1})\times \cdots \times (\Gamma_h\circ \Pi^S_{i_r})\\  
                                                           &=  {\displaystyle\sum_{i_1+\cdots+i_r=i}} ( \Pi^S_{i_1}\circ \Gamma_h)\times \cdots \times ( \Pi^S_{i_r}\circ \Gamma_h)\\  
                                                           &= {\displaystyle\sum_{i_1+\cdots+i_r=i}} (\Pi_{i_1}^S\times \cdots\times \Pi_{i_r}^S)      \circ (\Gamma_h\times\cdots\times\Gamma_h)\\
                                                           &= \Pi_i^{S^r}\circ \Gamma_{h^{\times r}}\ \ \ \hbox{in}\ A^{2r}(S^r\times S^r)\ .\\
                                                           \end{split}\]
                Here, the first and last lines are the definition of the product MCK decomposition for $S^r$; the second and fourth line are just regrouping, and the third line is lemma \ref{comm}.                                           
       \end{proof}
  
 \begin{lemma}\label{idempx}  Let $S$ be an algebraic $K3$ surface, and $G_S\subset\aut(S)$ a group of finite order $k$. For any $r\in\NN$, let $X=S^{[r]}$ and let $G\subset\aut(X)$ be the group of natural automorphisms induced by $G_S$. Let $\{\Pi_i^X\}$ be the MCK decomposition of theorem \ref{hilbk}. Let $\Delta^G_X$ denote the correspondence
   \[ \Delta^G_{X}:={1\over k}{\displaystyle\sum_{g\in G}}\ \Gamma_g\ \ \ \in A^{2r}(X\times X)\ .\]
  Then
  \[ \Delta^G_{X}\circ \Pi_i^{X}= \Pi_i^{X}\circ \Delta^G_{X}\ \ \ \in A^{2r}(X\times X)\ \]
  is an idempotent, for any $i$.
  \end{lemma}
  
 \begin{proof} Again, it suffices to prove the commutativity statement. This can be done as follows: for any $g\in G$, we can write $g=h^{[r]}$ where $h\in\aut(S)$. Then we have 
    \begin{equation}\label{gX} \begin{split}  \Gamma_g\circ \Pi_i^X &= \Gamma_g\circ {\displaystyle\sum_{\mu\in \Bb(k)}} \ {1\over m_\mu} {\Gamma}_\mu \circ \Pi^{S^{\mu}}_{i-2k+2l(\mu)}\circ {}^t {\Gamma}_\mu\\
                                   &=   {\displaystyle\sum_{\mu\in \Bb(k)}} \ {1\over m_\mu}\Gamma_g\circ {\Gamma}_\mu \circ \Pi^{S^{\mu}}_{i-2k+2l(\mu)}\circ {}^t {\Gamma}_\mu\\
                                   &=  {\displaystyle\sum_{\mu\in \Bb(k)}} \ {1\over m_\mu} {\Gamma}_\mu \circ \Gamma_{h^{\times l(\mu)}} \circ \Pi^{S^{\mu}}_{i-2k+2l(\mu)}\circ {}^t {\Gamma}_\mu\\
                                   & =  {\displaystyle\sum_{\mu\in \Bb(k)}} \ {1\over m_\mu} {\Gamma}_\mu \circ \Pi^{S^{\mu}}_{i-2k+2l(\mu)}\circ \Gamma_{h^{\times l(\mu)}} \circ {}^t {\Gamma}_\mu\\
                                   & =  {\displaystyle\sum_{\mu\in \Bb(k)}} \ {1\over m_\mu} {\Gamma}_\mu \circ \Pi^{S^{\mu}}_{i-2k+2l(\mu)}\circ {}^t {\Gamma}_\mu\circ \Gamma_g\\
                                   &= \Pi_i^X\circ \Gamma_g\ \ \ \hbox{in}\ A^{2r}(X\times X)\ .\\
                                   \end{split}\end{equation}
                        Here, the first line follows from the definition of $\Pi_i^X$ (definition (\ref{defv})). The second line is just regrouping, the third line is by construction of natural automorphisms of $X$, the fourth line is equality (\ref{eachg}) above, and the fifth line is again by construction of natural automorphisms.          
                                   \end{proof} 
      
 \begin{remark}\label{grade0} In view of \cite[Lemma 1.4]{SV2} the commutativity property (\ref{gX}) is equivalent to the following: for any natural automorphism $g$ of $X$, the graph $\Gamma_g\in A^{2r}(X\times X)$ is ``of pure grade $0$'', i.e. $\Gamma_g\in A^{2r}_{(0)}(X\times X)$.
 \end{remark}
         
\begin{lemma}\label{quotientmck} Let $S$ be an algebraic $K3$ surface, and $X=S^{[m]}$ the Hilbert scheme of length $m$ subschemes. Let $G\subset\aut(X)$ a group of finite order $k$ of natural automorphisms. Then the quotient $Y:=X/G$ has a self--dual MCK decomposition.
\end{lemma}

\begin{proof} Let $p\colon X\to Y$ denote the quotient morphism. One defines
    \[ \Pi_j^Y:= {1\over k} \Gamma_p\circ \Pi_j^X\circ {}^t \Gamma_p\ \ \ \in A^{2m}(Y\times Y)\ ,\]
    where $\{\Pi^X_j\}$ is the self--dual MCK decomposition of theorem \ref{hilbk}.
    This defines a self--dual CK decomposition $\{\Pi_j^Y\}$, since
    \[   \begin{split} \Pi_i^Y\circ \pi_j^Y &= {1\over k^2} \Gamma_p\circ \Pi_i^X\circ {}^t \Gamma_p  \circ   \Gamma_p\circ \Pi_j^X\circ {}^t \Gamma_p  \\
                         &= {1\over k} \Gamma_p\circ \Pi_i^X\circ \Delta^G_X \circ \Pi_j^X\circ {}^t \Gamma_p \\
                         &= {1\over k} \Gamma_p\circ \Pi_i^X\circ \Pi_j^X\circ \Delta^G_X\circ {}^t \Gamma_p \\
                         &=\begin{cases}  0 &\hbox{if\ }i\not=j\ ;\\
                                                    {1\over k} \Gamma_p\circ \Pi_i^X\circ {}^t \Gamma_p=\Pi_i^Y &\hbox{if\ }i=j\ .\\
                                                   \end{cases} 
                                              \end{split}\]   
           (Here, in the third line we have used lemma \ref{comm}.)
           
   It remains to check this CK decomposition is multiplicative. To this end, let $i,j,k$ be integers with $k\not=i+j$. We note that                                 
           \[    \begin{split}    
      \Pi^Y_k\circ \Delta_Y^{sm}\circ (\Pi^Y_i\times \Pi^Y_j)  &= {1\over k^3}\ \ \Gamma_p\circ \Pi^X_k\circ {}^t \Gamma_p\circ \Delta^Y_{sm}\circ \Gamma_{p\times p}\circ (\Pi^X_i\times \Pi^X_j)\circ {}^t \Gamma_{p\times p}\\
                          &=\Gamma_p\circ \Pi^X_k\circ \Delta^G_X\circ \Delta_X^{sm}\circ (\Delta^G_X\times \Delta^G_X)\circ   (\Pi^X_i\times \Pi^X_j)\circ {}^t \Gamma_{p\times p}\\
                       & = \Gamma_p\circ \Delta^G_X\circ \Pi^X_k\circ \Delta_X^{sm}\circ   (\Pi^X_i\times \Pi^X_j)\circ (\Delta^G_X\times \Delta^G_X)\circ {}^t \Gamma_{p\times p}\\
                       & =0\ \ \ \hbox{in}\ A^{4m}(Y\times Y\times Y)\ .\\
                       \end{split}\]
                 Here, the first equality is by definition of the $\Pi^Y_i$, the second equality is lemma \ref{sm} below, the third equality follows from lemma \ref{idempx}, and the fourth equality is the fact that the $\{\Pi^X_i\}$ are an MCK decomposition for $X$.      
        
     \begin{lemma}\label{sm} There is equality
     \[ \begin{split} {}^t \Gamma_p\circ \Delta_Y^{sm}\circ \Gamma_{p\times p}&= (\sum_{g\in G} \Gamma_g)\circ \Delta_X^{sm}\circ \bigl((\sum_{g\in G} \Gamma_g)\times 
         (\sum_{g\in G}           \Gamma_g)\bigr) \\  
         &= k^3\ \Delta^G_X\circ \Delta_X^{sm}\circ (\Delta^G_X\times \Delta^G_X)\ \ \ \hbox{in}\ A^{4m}(X\times X\times X)\ .\\
         \end{split}\]                    
                                                \end{lemma}
                                                
  \begin{proof} The second equality is just the definition of $\Delta^G_X$. As to the first equality, we first note that
         \[  \Delta_Y^{sm}  =(p\times p\times p)_\ast (\Delta_X^{sm}) = \Gamma_p\circ \Delta_X^{sm}\circ {}^t \Gamma_{p\times p}\ \ \ \hbox{in}\ A^{4m}(Y\times Y\times Y)\ .\]
                 This implies that
                 \[ {}^t \Gamma_p\circ \Delta_Y^{sm}\circ \Gamma_{p\times p}= {}^t \Gamma_p\circ \Gamma_p\circ \Delta_X^{sm}\circ {}^t \Gamma_{p\times p}\circ \Gamma_{p\times p}\ .\]
                 But ${}^t \Gamma_p\circ \Gamma_p=\sum_{g\in G} \Gamma_g$, and thus
                 \[ {}^t \Gamma_p\circ \Delta_Y^{sm}\circ \Gamma_{p\times p}=  (\sum_{g\in G} \Gamma_g)\circ \Delta_X^{sm}\circ \bigl((\sum_{g\in G} \Gamma_g)\times 
         (\sum_{g\in G}           \Gamma_g)\bigr) \ \ \ \hbox{in}\ A^{4m}(X\times X\times X)\ ,\]
         as claimed. 
         \end{proof}
  \end{proof}

   There is also the following commutativity relation:            

\begin{lemma}\label{comprod} Let 
 \[ \Xi_1\ ,\ldots, \ \Xi_m\in  A^{2}(S\times S^{m})  \]
 be as in propositions \ref{prod} and \ref{prod3}. Let $h\in\aut(S)$. The diagrams
    \[ \begin{array}[c]{ccc}
           A^{2}_{(2)}(S^m)& \xrightarrow{(({}^t\Xi_1\vert_{S^{m+1}})_\ast,\ldots, ({}^t\Xi_m\vert_{S^{m+1}})_\ast)}&
             A^2_{(2)}(S)\oplus \cdots \oplus A^2_{(2)}(S) \\
               &&\\
            \ \ \ \ \ \ \ \ \  \downarrow{(h^{\times m})_\ast}&&\ \ \ \ \ \ \ \ \ \ \ \ \ \ \ \  \downarrow{ (h_\ast,\ldots,h_\ast)}\\
            &&\\
            A^{2}_{(2)}(S^m)& \xrightarrow{(({}^t\Xi_1\vert_{S^{m+1}})_\ast,\ldots, ({}^t\Xi_m\vert_{S^{m+1}})_\ast)}&
             A^2_{(2)}(S)\oplus \cdots \oplus A^2_{(2)}(S) \\
          \end{array}\]
and
        \[ \begin{array}[c]{ccc}
            A^2_{(2)}(S)\oplus \cdots \oplus A^2_{(2)}(S)&
              \xrightarrow{(\Xi_1)_\ast+\ldots+(\Xi_m)_\ast}& A^{2m}_{(2)}(S^m)  \\
              &&\\
             \ \ \ \ \ \ \ \  \ \ \ \ \ \ \ \ \ \downarrow{ (h_\ast,\ldots,h_\ast)} &&     \ \ \ \ \  \ \ \ \downarrow{(h^{\times m})_\ast}\\
             &&\\
              A^2_{(2)}(S)\oplus \cdots \oplus A^2_{(2)}(S)&
              \xrightarrow{(\Xi_1)_\ast+\ldots+(\Xi_m)_\ast}& A^{2m}_{(2)}(S^m)  \\
          \end{array}\]
are commutative. 
  \end{lemma}
  
  \begin{proof} First, we observe that $h^{\times m}$ and $h$ preserve the bigrading in view of (\ref{eachg}), so the diagrams make sense. Next, we recall (proposition \ref{prod}) that
  ${}^t \Xi_i$ is defined on $A^2_{(2)}(S^m)$ as projection on the $i$th factor (which also preserves the bigrading, cf. \cite[Corollary 1.6]{SV2}). The commutativity of the first diagram now follows from the commutativity of
  \[ \begin{array}[c]{ccc}
    S^m & \xrightarrow{p_i}& S\\
    \ \ \ \downarrow{ h^{\times m}} && \ \ \ \downarrow{ h}\\
    S^m & \xrightarrow{p_i}& S\\
    \end{array}\]
  As for the second diagram: $\Xi_i$ acts on $A^2_{(2)}(S)$ as $(p_i)^\ast$. Since we can write $h_\ast=(h^{-1})^\ast$, the second diagram is also commutative.  
      \end{proof}

 
 

\subsection{Natural automorphisms of Hilbert schemes}

\begin{definition}[Boissi\`ere \cite{Bo}] Let $S$ be a surface, and let $X=S^{[m]}$ denote the Hilbert scheme of length $m$ subschemes. An automorphism $\psi\in\aut(S)$ induces an automorphism $\psi^{[m]}$ of $X$. This determines a homomorphism
  \[ \begin{split} \aut(S)\ &\to\ \aut(X)\ ,\\
              \psi\ &\mapsto\ \psi^{[m]}\ ,\\
           \end{split}\]
    which is injective \cite{Bo}. The image of this homomorphism is called the group of {\em natural automorphisms\/} of $X$.          
\end{definition}



\begin{remark} It is known \cite[Theorem 1]{BoSa} that an automorphism of a Hilbert scheme is natural if and only if it fixes the exceptional divisor of the Hilbert--Chow morphism.
To find examples of non--natural automorphisms of a Hilbert scheme $X$, Boissi\`ere and Sarti introduce the notion of {\em index\/} of an automorphism of $X$. For Hilbert schemes of a generic algebraic $K3$ surface, the index of an automorphism is $1$ if and only if the automorphism is natural \cite[section 4]{BoSa}.
\end{remark}

\section{Main result}

This section contains the proof of the main result of this note, theorem \ref{main}.

\begin{definition} Let $S$ be a $K3$ surface, and let $h\in\aut(S)$ be an automorphism of order $k$. We say that $h$ is {\em non--symplectic\/} if 
  \[ h^\ast=\nu\cdot \ide\colon\ \ \ H^{2,0}(S)\ \to\ H^{2,0}(S)\ ,\]
  where $\nu$ is a primitive $k$--th root of unity.
  
  (NB: this is sometimes referred to as a ``purely non--symplectic automorphism''.)
 \end{definition}

\begin{theorem}\label{main} 
Let $S$ be a projective $K3$ surface, and let $X=S^{[m]}$ be the Hilbert scheme of length $m$ subschemes. Let $G\subset\aut(X)$ be a subgroup of order $k$ of natural non--symplectic automorphisms. Then
  \[ A^i_{(2)}(X)\cap A^i(X)^G =0\ \ \ \hbox{for}\ i\in\{2,2m\}\ .\]
\end{theorem}

\begin{proof} Let us start with the case $i=2$, i.e. codimension $2$ cycles. 
To prove the required vanishing 
  \[A^2_{(2)}(X)\cap A^2(X)^G =0\] 
  is equivalent to showing that
  \begin{equation}\label{corr2}  \bigl( \Delta^G_X\circ \Pi_2^X\bigr){}_\ast =0\ \colon\ \ \ A^2(X)\ \to\ A^2(X)\ ,\end{equation}
  where $\Pi_2^X$ is part of an MCK decomposition for $X$.
  
  As we have seen (remark \ref{compat}, plus the obvious fact that $A^1_{(2)}(S^{m-1})=0$), there is a commutative diagram
  \[ \begin{array}[c]{ccc}
       A^2_{(2)}(X) & \hookrightarrow & A^2_{(2)}(S^m)\\
        &&\\
          \ \ \ \  \downarrow {\scriptstyle (\Delta_X^G)_\ast} &&\ \ \ \downarrow {\scriptstyle (\Delta^G_{S^m})_\ast} \\
          &&\\
     A^2_{(2)}(X) & \hookrightarrow & A^2_{(2)}(S^m)\\
     \end{array}\]
   where horizontal arrows are split injective. Here, the correspondence $\Delta^G_{S^m}$ is defined as
     \[  \Delta^G_{S^m}:= {\displaystyle\sum_{h\in G_{S}}} \Gamma_h\times\Gamma_h\times\cdots\times\Gamma_h\in A^{2m}(S^m\times S^m)\ ,\]   
  and the diagram commutes because of the construction of natural automorphisms of $X$.   
     
     To prove (\ref{corr2}), we are thus reduced to proving that
      \begin{equation}\label{corr3}  \bigl( \Delta^G_{S^m}\circ \Pi_2^{S^m}\bigr){}_\ast =0\ \colon\ \ \ A^2(S^m)\ \to\ A^2(S^m)\ ,\end{equation}
      where $\Pi_2^{S^m}$ is part of an MCK decomposition $\{\Pi_i^{S^m}\}$ for $S^m$. We will suppose $\{\Pi_i^{S^m}\}$ is the product MCK decomposition used in the proof of theorem \ref{hilbk}.   
      
   We state a lemma:

\begin{lemma}\label{lemma2} The surface $R:=S/G_S$ has
  \[ A^2_{hom}(R)=0\ .\]
  Equivalently, for any MCK decomposition $\{\Pi^S_j\}$ one has
  \[ (\Delta^G_S\circ \Pi_2^S)_\ast=0\colon\ \ \ A^2(S)\ \to\ A^2(S)\ .\]
   \end{lemma}
 
   \begin{proof} The quotient variety $R$ has geometric genus $0$. Since quotient singularities are rational singularities, there exists a resolution $Y\to R$ with $p_g(Y)=0$. Since $Y$ is not of general type, Bloch's conjecture is known to hold for $Y$ \cite{BKL}, i.e. $A^2_{hom}(Y)=0$. This implies that also $A^2_{hom}(R)=0$.
 \end{proof}
 
Armed with this lemma, we can prove the vanishing (\ref{corr3}): 
There is a commutative diagram
  \[ \begin{array}[c]{ccc}
           A^{2}_{(2)}(S^m)& \xrightarrow{(({}^t\Xi_1\vert_{S^{m+1}})_\ast,\ldots, ({}^t\Xi_m\vert_{S^{m+1}})_\ast)}&
             A^2_{(2)}(S)\oplus \cdots \oplus A^2_{(2)}(S) \\
               &&\\
            \ \ \ \ \ \ \ \ \  \downarrow{(\Delta^G_{S^m})_\ast}&&\ \ \ \ \ \ \ \ \ \ \ \ \ \ \ \  \downarrow{ ((\Delta^G_S)_\ast,\ldots,(\Delta^G_S)_\ast)}\\
            &&\\
            A^{2}_{(2)}(S^m)& \xrightarrow{(({}^t\Xi_1\vert_{S^{m+1}})_\ast,\ldots, ({}^t\Xi_m\vert_{S^{m+1}})_\ast)}&
             A^2_{(2)}(S)\oplus \cdots \oplus A^2_{(2)}(S) \\
          \end{array}\]
   The commutativity of this diagram is lemma \ref{comprod}. Horizontal arrows are injections thanks to proposition \ref{prod3}. Since the right vertical arrow is the zero map (lemma \ref{lemma2}), the left vertical arrow is also the zero map; this proves the vanishing (\ref{corr3}).

 The statement for $i=2m$ is proven similarly: 
 in view of remark \ref{compat}, there is a commutative diagram
     \[ \begin{array}[c]{ccc}
       A^{2m}_{(2)}(X) & \hookrightarrow & A^{2m}_{(2)}(S^m)\\
        &&\\
          \ \ \ \ \ \ \  \downarrow {\scriptstyle (\Delta_X^G)_\ast} &&\ \ \ \downarrow {\scriptstyle (\Delta^G_{S^m})_\ast} \\
          &&\\
     A^{2m}_{(2)}(X) & \hookrightarrow & A^{2m}_{(2)}(S^m)\\
     \end{array}\]
   where horizontal arrows are split injective. It thus suffices to prove the right vertical arrow is the zero map.
 
 Thanks to proposition \ref{prod} and lemma \ref{comprod}, there is a commutative diagram
   \[ \begin{array}[c]{ccc}
            A^2_{(2)}(S)\oplus \cdots \oplus A^2_{(2)}(S)&
              \xrightarrow{(\Xi_1)_\ast+\ldots+(\Xi_m)_\ast}& A^{2m}_{(2)}(S^m)  \\
              &&\\
             \ \ \ \ \ \ \ \  \ \ \ \ \ \ \ \ \ \downarrow{ ((\Delta^G_S)_\ast,\ldots,(\Delta^G_S)_\ast)} &&     \ \ \ \ \  \ \ \ \downarrow{(\Delta^G_{S^m})_\ast}\\
             &&\\
              A^2_{(2)}(S)\oplus \cdots \oplus A^2_{(2)}(S)&
              \xrightarrow{(\Xi_1)_\ast+\ldots+(\Xi_m)_\ast}& A^{2m}_{(2)}(S^m)  \\
          \end{array}\]
          where horizontal arrows are surjections. Combined with lemma \ref{lemma2}, this settles the $i=2m$ case.  
                                       \end{proof}

\begin{remark} Let $X$ and $G$ be as in theorem \ref{main}. Let $X^\prime$ be a hyperk\"ahler variety birational to $X$, and let $G^\prime$ be the group of birational self--maps of $X^\prime$ induced by $G$.
Applying proposition \ref{birat}, it follows from theorem \ref{main} that also
  \[  A^i_{(2)}(X^\prime)\cap A^i(X^\prime)^{G^\prime} =0\ \ \ \hbox{for}\ i\in\{2,2m\}\ .\]
\end{remark}

\section{Some corollaries}

\begin{corollary}\label{cor1} Let $X$ and $G$ be as in theorem \ref{main}, and let $Y:=X/G$ be the quotient. For any $r\in\NN$, let
  \[ E^\ast(Y^r)\ \subset\ A^\ast(Y^r) \]
  be the subalgebra generated by (pullbacks of) $A^1(Y)$ and $A^2(Y)$ and $\Delta_Y$, $\Delta_Y^{sm}$. Then the cycle class map induces maps
  \[ E^i(Y^r)\ \to\ H^{2i}(Y^r) \]
  that are injective for $i\ge 2mr-1$.
  \end{corollary}
  
  \begin{proof} First, it follows from lemma \ref{quotientmck} that $Y$, and hence $Y^r$, has a self--dual MCK decomposition. Consequently, the Chow ring $A^\ast(Y^r)$ is a bigraded ring. Theorem \ref{main} (plus the obvious fact that $A^1_{hom}(Y)=0$) implies that
  \[  A^i(Y)= \bigoplus_{j\le 0}  A^i_{(j)}(Y)\ \ \ \hbox{for}\ i\le 2\ .\]
  Lemma \ref{diag} ensures that
  \[ \Delta_Y\in A^{2m}_{(0)}(Y\times Y)\ ,\ \ \ \Delta_Y^{sm}\in A^{4m}_{(0)}(Y^3)\ .\]
  Since pullbacks for projections of type $Y^r\to Y^{s}, s<r$, preserve the bigrading (this follows from \cite[Corollary 1.6]{SV2}, or alternatively can be checked directly), this implies that
  \[ E^\ast(Y^r)\subset  \bigoplus_{j\le 0} A^\ast_{(j)}(Y^r)\ .\]
 The corollary now follows from the fact that
  \[  A^i_{(j)}(Y^r)\ \to\  A^i_{(j)}(X^r) \]
  is injective (this is true for any $i$ and $j$), and the fact that
  \[ \begin{split}  &A^i_{(j)}(X^r)=0\ \ \ \hbox{for}\ i\ge 2mr-1\ \hbox{and}\ j<0\ ,\\
     &A^i_{(0)}(X^r)\cap A^i_{hom}(X^r) =0\ \ \ \hbox{for}\ i\ge 2mr-1\\
     \end{split} \ \]
  (as noted in \cite[Introduction]{V6}).
    \end{proof}
  
 \begin{corollary}\label{cor2} Let $X$ and $G$ be as in theorem \ref{main}, and let $Y:=X/G$ be the quotient. Let $a\in A^{2m}(Y)$ be a $0$--cycle which is in the image of the intersection product map
   \[  A^3(Y)\otimes A^{i_1}(Y)\otimes\cdots\otimes A^{i_s}(Y)\ \to\ A^{2m}(Y)\ ,\]
   with all $i_j\le 2$ (and $i_1+\cdots+i_s=2m-3$). Then $a$ is rationally trivial if and only if $\deg(a)=0$.
   \end{corollary}

\begin{proof} The point is that
   \[  \begin{split} A^3(Y)&= \bigoplus_{r\le 2}  A^3_{(r)}(Y)\ ,\\
                           A^{i_j}(Y)&=\bigoplus_{r\le 0} A^{i_j}_{(r)}(Y)\ \ \ \hbox{for}\ i_j\le 2\ \\
            \end{split}\]
         (theorem \ref{main}),   
     and so
     \[ a\in \bigoplus_{r\le 2} A^{2m}_{(r)}(Y)\ .\]
     But we know that $A^{2m}_{(r)}(Y)=0$ for $r<0$ (this is a general fact for any variety with an MCK decomposition), and we have seen that $A^{2m}_{(2)}(Y)=0$ (theorem \ref{main}), and so 
     \[         a\in A^{2m}_{(0)}(Y)\cong\QQ\ .\]
    \end{proof}

 \begin{remark} Results similar to corollaries \ref{cor1} and \ref{cor2} have been obtained for $0$--cycles on certain Calabi--Yau varieties. 
  If $Y$ is a Calabi--Yau variety (of dimension $n$) that is a generic complete intersection in projective space, it is known that the image of the intersection product
  \[ \ima\Bigl(  A^i(Y)\otimes A^{n-i}(Y)\ \to\ A^n(Y) \Bigr)\ ,  \ \ 0<i<n\ ,\]
  is of dimension $1$, and hence injects into cohomology \cite{V13}, \cite{LFu}. 
  
  Going beyond the Calabi--Yau case, there is also a result of L. Fu for generic hypersurfaces $Y$ of {\em general type\/}. Here, the image of the intersection product
    \[ \ima\Bigl(  A^{i_1}(Y)\otimes A^{i_2}(Y)\otimes\cdots\otimes A^{i_m}(Y)\ \to\ A^n(Y) \Bigr)\ ,  \ \ i_j>0\ ,\]
  is again of dimension $1$, provided $m$ is large enough relative to the degree of $Y$ \cite[Theorem 2.13]{LFu}. This is very similar to the behaviour of the Chow ring exhibited in corollary \ref{cor2}.
   \end{remark}






\vskip0.6cm

\begin{acknowledgements} Thanks to Len, Kai and Yasuyo for numerous pleasant coffee breaks. 
\end{acknowledgements}

\end{document}